\documentclass{amsart}
\usepackage[english]{babel}
\usepackage{amsmath,amsfonts,amstext,amsbsy,amssymb,amsthm,flafter}
\usepackage{multirow}
\usepackage{url}
\usepackage{leftidx}
\usepackage{fixltx2e}

\usepackage[all]{onlyamsmath}
\usepackage{amsrefs}
\newtheorem{theorem}{Theorem}

\theoremstyle{plain}
\newtheorem{lemma}{Lemma}

\theoremstyle{remark}
\newtheorem{remark}{Remark}

\newcommand{\C}{\ensuremath{\mathbb{C}}}
\newcommand{\Z}{\ensuremath{\mathbb{Z}}}
\newcommand{\N}{\ensuremath{\mathbb{N}}}
\newcommand{\R}{\ensuremath{\mathbb{R}}}
\newcommand{\A}{\ensuremath{\mathcal{A}}}
\newcommand{\B}{\ensuremath{\mathcal{B}}}
\newcommand{\E}{\ensuremath{\mathcal{E}}}
\newcommand{\T}{\ensuremath{\mathbb{T}}}
\newcommand{\Sc}{\ensuremath{\mathcal{S}}}
\newcommand{\Hi}{\ensuremath{\mathcal{H}}}
\newcommand{\At}{\A (\T_\theta^n)}

\newcommand{\vmu}{\boldsymbol{\mu}}

\newcommand{\vx}{\mathbf{x}}
\newcommand{\vy}{\mathbf{y}}

\newcommand{\vtau}{\boldsymbol{\tau}}

\newcommand{\vdelta}{\boldsymbol{\delta}}

\newcommand{\cd}{\cdot}
\newcommand{\SkewM}{\ensuremath{\text{Skew}_n (\mathbb{Z})}}

\newcommand{\pmJJ}{\epsilon_J}

\newcommand{\Too}{\mathbf{T_{11}}}
\newcommand{\Ttt}{\mathbf{T_{32}}}

\makeatletter
\def\url@leostyle{%
  \@ifundefined{selectfont}{\def\UrlFont{\sf}}{\def\UrlFont{\small\ttfamily}}}
\makeatother
\def\XXint#1#2#3{{\setbox0=\hbox{$#1{#2#3}{\int}$}
     \vcenter{\hbox{$#2#3$}}\kern-.5\wd0}}

\let\l\left
\let\r\right

\title{Morita ``equivalences'' of equivariant torus spectral triples}
\author{Jan Jitse Venselaar}
\date{28/12/2011}
\address{Mathematical Institute\\
         Utrecht University\\
	 PO Box 80010, 3508 TA Utrecht\\
	 The Netherlands}
\email{J.J.Venselaar1@uu.nl}
 \keywords{spectral geometry, noncommutative geometry, equivariant spectral triples, Morita equivalences of spectral triples}
\subjclass[2010]{58B34; 46L87}

\begin{document}

\begin{abstract}
In general, Morita equivalence of spectral triples need not be a symmetric relation.
In this paper, we show that Morita equivalence of spectral triples is an equivalence relation for equivariant torus spectral triples.
\end{abstract}

\maketitle

\section{Introduction}
Two algebras are Morita equivalent if their categories of representations are equivalent. If the algebras are commutative, this implies that they are isomorphic, but for noncommutative algebras, the notion is more general.

By the classical Gelfand-Naimark duality, Morita equivalences of $C^*$-algebras can be thought of as homeomorphisms of the corresponding Hausdorff spaces. The geometry of a manifold can be encoded by a so-called spectral triple $(\A,\Hi,D)$, through the interaction via a Hilbert space $\Hi$ of a commutative $C^*$-algebra $\A$, characterizing the topology, and the Dirac operator $D$, characterizing the metric. The spin geometry of a manifold, a refinement of Riemannian geometry, is encoded through adding to this data an antilinear isometry $J$ of the Hilbert space, called ``reality operator''. The conditions for a spectral triple can be relaxed to include noncommutative $C^*$-algebras, leading to noncommutative geometry. 

In order to have an isomorphism, or isometry, of spectral triples, we thus require the algebras to be Morita equivalent, and modify the Hilbert space and Dirac operator in such a way that all geometric data is preserved~\cite{connes_gravity_1996}. If the algebra is noncommutative, this gives rise to interesting new possibilities for ``diffeomorphisms'' which do not occur if the algebra is commutative.

For example, given a spectral triple where the algebra is the algebra of smooth functions on a four-manifold times a certain finite dimensional matrix algebra, the Morita self-equivalences of the algebra change the Dirac operator. These changes take the form of connections, which for suitably chosen matrix algebras constitute the gauge group of the Standard Model of particle physics~\cite{chamseddine_noncommutative_2010}.

However, a Morita equivalence of spectral triples (see Section~\ref{sec:morita equivalences}) is not a true equivalence relation, as it is not symmetric in general, see for example~\cite{MR2371808}*{Remark 1.143}. In this paper, using the results of~\cite{rieffel_morita_1999} and~\cite{elliott_morita_2007}*{Theorem 1.1} on Morita equivalences of the algebra of the noncommutative torus, we show what if we restrict to equivariant spectral triples for the noncommutative $n$-torus (as classified by the author in~\cite{venselaar_2010}) Morita equivalences are symmetric. This appears to be the first non-trivial example where symmetry can be proven. The Morita equivalences also lead to new isometries between real spectral triples which are not present in the commutative case, even though all such spectral triples were shown to be isospectral deformations of commutative tori.

\section{Torus equivariant real spectral triples}\label{sec:equivariant spectral triples}
All possible smooth real spectral triples which are equivariant with respect to an $n$-torus action were classified by the author in~\cite{venselaar_2010}. We will give a short recap of the relevant results and notation. 
A smooth spectral triple is given by the data $(\A,\Hi,D)$, with $\A$ a pre-$C^*$ algebra, here assumed to be unital and separable, $\Hi$ a separable Hilbert space on which $\A$ has a representation $\pi$, acting as bounded operators. The operator $D$, usually called Dirac operator, is an unbounded operator on $\Hi$. This triple has to satisfy certain conditions, clarified for example in~\cite{vrilly_introduction_2006}. A real spectral triple has an extra operator $J$, an antilinear isometry of $\Hi$ with itself.

An equivariant spectral triple is a spectral triple for which a Hopf algebra acts on all of the data of spectral triple in a compatible way, as described in~\cite{sitarz_equivariant_2003}.

The algebra of an equivariant real spectral triple on the noncommutative $n$-torus is generated by unitary elements $U_{\vx}$ with $\vx\in\Z^n$ such that
\[ U_{\vx} U_{\vy} = e(\vx\cd\theta\vy) U_{\vy} U_{\vx},\]
with $\theta$ a real skew-symmetric $n\times n$ matrix. 
As shorthand we write $e(\cd) = e^{2\pi i \cd}$, and $\vx\cd\vy$ for the standard inner product for vectors in $\R^n$.

The Hopf algebra for which the spectral triple of the noncommutative torus is equivariant, is the algebra $U(\mathfrak{t}^n)$, the universal enveloping algebra of the Lie algebra of the $n$-torus. The Lie algebra is generated by $n$ generators $\delta_i$ such that $\delta_i \delta_j = \delta_j \delta_i$. The algebra has a representation $\rho$ on the Hilbert space, a representation $\phi$ on the $C^*$-algebra such that $\phi(\delta_i) U_{\vx} = \vx_i U_{\vx}$, and the representations satisfy:
\begin{subequations}
 \begin{align}
%  h_1 h_2 &= h_2 h_1 \label{eqn:commutative hopf algebra},\\
 \rho(h)\pi(a) v &= \l(\pi( \phi(h) a ) + \pi(a) \rho(h)\r) v\label{eqn:equivariance algebra},\\
 \rho(h) D v &=  D \rho(h) v\label{eqn:equivariance dirac operator},\\
 \end{align}
\end{subequations}
for $v\in \Hi$, $h\in U(\mathfrak{t}^n)$ and $a\in\A$. The real structure $J$ anticommutes with the Hopf algebra action.

Any set of mutual eigenvectors $e_{\vmu}$ of the generators $\delta_i$, with $\delta_i e_{\vmu} = \vmu_i e_{\vmu}$, $\vmu\in \R^n$, form the basis of a Hilbert space on which $U(\mathfrak{t}^n)$ acts in accordance with the above conditions. 
If the set of eigenvectors $e_{\vmu}$ is of the form ${\{e_{\vmu_0 + \vx}\}}_{\vx\in\Z^n}$, the algebra $\A$ acts naturally on this space.
% , by $U_{\vx} e_{\vmu} = e(\vx\cd\am\vmu + \vx\cd\am\vmu) e_{\vmu + \vx}$, with $\am$ any $n\times n$ matrix such that $\am - \am^t = \theta$.

In order for the real spectral triple to satisfy all conditions as stated for example in~\cite{connes_noncommutative_1995}, we restrict the vectors $\vmu$ to lie either in the lattice $\Z^n$, or $\Z^n$ shifted by $1/2$ in one or more directions.
Furthermore we take $2^{\lfloor n/2\rfloor}$ copies of the Hilbert space. The algebra $\A$ acts diagonally on this space. The Dirac operator is given by
\begin{equation}D = \sum_i (\vtau_i \cd\vdelta) A_i +B,\label{eqn:equivariant Dirac operator}\end{equation} with ${\{\vtau_i\}}_{i=1}^n$ linearly independent vectors spanning $\R^n$, $B$ a bounded self-adjoint operator commuting with the algebra, and ${\{A_i\}}_{i=1}^n$ a set of $2^{\lfloor n/2\rfloor} \times 2^{\lfloor n/2\rfloor}$ self-adjoint matrices which generate an irreducible representation of the Clifford algebra $Cl_{n,0}$ on $\C^{\lfloor n/2\rfloor}$. The real structure $J$ is then uniquely determined, up to multiplication with a complex number of norm $1$.

\section{Morita equivalences of spectral triples}\label{sec:morita equivalences}
\subsection{Definitions: general case}
Two pre-$C^*$ algebras $\A$ and $\A'$ are strongly Morita equivalent if there exists a Hilbert $C^*$-bimodule $\leftidx{_{\A'}}{\E}_\A$ for which we have $\A'=\text{End}_{\A} (\E)$~\cite{MR0353003}. If $\A$ and $\A'$ are strongly Morita equivalent, we say that the spectral triple $(\A,\Hi,D)$ is Morita equivalent to $(\A',\Hi',D')$, if $\Hi' = \E\otimes_{\A} \Hi$ and the Dirac operators $D$ and $D'$ are related by a connection.
The connection is a way to transport the action of $D$ from $\Hi$ to $\Hi'= \E\otimes_{\A} \Hi$. This is needed, since the action of $D$ does not in general commute with $\A$.

%  of this construction is how to define the Dirac operator $D'$ on $\Hi' = \E\otimes_{\A} \Hi$. Since the action of $D$ does not in general commute with $\A$, we need a way to transport
The \emph{connection} $\nabla_D$ on $\E$ is a linear operator on $\E$ such that
\begin{equation} D' (s \otimes\xi) = \nabla(s) \xi+ s\otimes D\xi,\label{eqn:new Dirac operator}\end{equation}
has all the right properties of a Dirac operator on $\Hi'$. Consistency of this definition with the action of $\A$ on $\E$ and $\Hi$ give that $\nabla_D$ must satisfy a \emph{Leibniz rule}:
\begin{equation}
 \nabla_D (s a) = \nabla(s) a + s \otimes [D,a],\label{eqn:Leibniz rule}
\end{equation}
for all $a\in \A$ and $s \in E$. In order to have consistency with the commutative case, where connections are given by one-forms acting by multiplication on the space of spinors, we have that $\nabla_D$ must take values in $\E \otimes \Omega_D^1$, where $\Omega_D^1$ is the space of bounded operators on $\Hi$ given by:
\begin{equation}
 \Omega_D^1 = \text{span}\{a[D,b] \mid a,b\in\A\}.\label{eqn:space of oneforms}
\end{equation}
By construction $\Omega_D^1$ is a bimodule over $\A$. The condition that $D'$ is a self-adjoint operator on $\Hi'$ gives an extra condition on $\nabla_D$, it must be \emph{Hermitian}:
\begin{equation}(r \vert\nabla_D s) - (\nabla_D r\vert s) = [D,(r\vert s)],\label{eqn:Hermitian condition} \end{equation}
for all $r,s\in \E$, where $(r\vert s)$ is the inner product of the Hilbert $C^*$-modules $\E$, taking values in $\A$.

From \eqref{eqn:new Dirac operator}, the Dirac operator $D'$ can be determined, up to a component commuting with the algebra $\A'$, as follows:
\begin{equation} [D',b] (s\otimes\psi) = [\nabla_D,b]s \otimes\psi+ bs \otimes D\psi - bs \otimes D\psi = [\nabla_D,b]s \otimes \psi.\label{eqn:new Dirac commutator}\end{equation}
We will denote a Morita equivalence of $(\A,\Hi,D)$ to $(\A',\Hi',D')$ by an equivalence bimodule $\E$ and connection $\nabla_D$ by $(\E,\nabla_D)$.

\subsection{Definition: real spectral triple} 
If we are considering a spectral triple with a real structure $J$, the Hilbert space $\Hi$ is an $\A$-bimodule, with the opposite algebra $\A^o = J \A J^{-1}$ acting from the right.
The Hilbert space $\Hi'$ can then be converted to a $\A'$-bimodule by setting $\Hi' = \E \otimes_{\A} H \otimes_{\A} \bar{\E}$ where $\bar{\E}$ is the opposite module $\leftidx{_\A}{\bar{\E}}_{\A'}$, and 
\begin{equation}D'(s\otimes \psi \otimes t) =  \nabla(s) \psi\otimes t + s \otimes D\psi \otimes t + s \otimes \psi (\overline{\nabla t}).\label{eqn:new Dirac operator real}\end{equation}
The real structure $J'$ is then given in the obvious way
\[J' (s\otimes \xi \otimes \bar{t}) = t \otimes J \xi \otimes \bar{s}.\]
We see that \eqref{eqn:new Dirac commutator} still can be used to calculate the new Dirac operator in this case.

\subsection{Examples}
If $\A$ is a noncommutative algebra, non-trivial examples are the Morita self-equivalences of $\A$ with itself. Take as equivalence bimodule the algebra $\A$ itself. The algebra and Hilbert space are unchanged, but the Dirac operator $D$ changes as~\cite{connes_gravity_1996}:
\begin{equation} D' = D + \mathbb{A} + \pmJJ J \mathbb{A} J^\dagger,\label{eqn:self-equivalence}\end{equation}
where $\mathbb{A}$ is given by a finite sum $\sum_j a_j [D,b_j]$, $a_j,b_j \in \A$. The sign $\pmJJ$ in the equation is $+1$ if the dimension of the spectral triple is $n\not\equiv 1 \text{ mod }4$ and $-1$ is $n\equiv 1\text{ mod }4$, where the dimension is a number describing the growth of the spectrum of the Dirac operator, see for example~\cite{connes_noncommutative_1995}. In the case of the noncommutative $n$-torus, the dimension is $n$.

The element $\mathbb{A}$ of $\Omega_D^1$ is called a gauge potential, and is an important part of the noncommutative geometry approach of the Standard Model of particle physics~\cite{chamseddine_noncommutative_2010}.

As a useful and interesting example, we compute a Morita equivalence of a commutative spectral triple $( C^\infty(S^1), L^2(S^1), i\dfrac{\partial}{\partial t})$, with the function algebra $C^\infty(S^1)$ generated by appropriately summable combinations of unitary generators $e^{2\pi i k t}$, $k\in \Z$. 

We take the equivalence bimodule to be $C^\infty(S^1)$. We have $\Omega_D^1 \simeq C^\infty(S^1)$, and so the connection $\nabla_D$ is a linear map from $C^\infty(S^1) \rightarrow C^\infty(S^1)$. The Leibniz rule \eqref{eqn:Leibniz rule} in this case reduces to $\nabla_D(ab) = \nabla(a)b + a [D,b]$, hence $\nabla_D(a)= [D,a] + c a$ with $c\in C^\infty(S^1)$. The Hermitian condition \eqref{eqn:Hermitian condition} then ensures $c = c^*$.
The new Dirac operator is then the same up to additive term commuting with the algebra, hence $D' = D + c$ with $c = c^* \in C^\infty(S^1)$. In the real case, since $J J^\dagger = 1$ and $JD = -DJ$ in dimension $1$, we see that 
\[J cJ^\dagger = J a[D,b] J^{\dagger} = a^* [JDJ^{\dagger},J bJ^{\dagger}] = -a^*[D,b^*] = -c^*,\]
hence $D' = D + c -c^* = D$. This is in fact true for all Morita equivalences of commutative real spectral triples when the Hilbert space is given as usual by $L^2$ sections of spinors, see the discussion in~\cite{MR1607870}*{Section 3.4}. If the Hilbert space is taken to be more general, commutative examples where $D'\neq D$ can be constructed, see~\cite{dongen_electrodynamics_2011}.

\subsection{Equivalence relation}
Morita equivalence is an equivalence relation for $C^*$-algebras, but what about Morita equivalences of spectral triples?
It is obviously a reflexive relation. Also not so hard to see is transitivity. For completeness we observe:

\begin{lemma}\label{lem:morita transitive}
 Morita equivalence of spectral triples is a transitive relation.
\end{lemma}
\begin{proof}
 Let $\A$, $\B$ and $\mathcal{C}$ be pre-$C^*$ algebras, such that $\A$ is strongly Morita equivalent to $\B$ via the bimodule $\leftidx{_{\B}}{\E}_{\A}$ and $\B$ strongly Morita equivalent to $\mathcal{C}$ via the bimodule $\leftidx{_{\mathcal{C}}}{\mathcal{F}}_{\mathcal{B}}$, one can form the bimodule $\leftidx{_{\mathcal{C}}}{\mathcal{G}}_{\A} = \mathcal{F}\otimes_{\B} \mathcal{E}$, and it is well known this is a Morita equivalence bimodule between $\A$ and $\mathcal{C}$.

Now given a Morita equivalence of spectral triples from $(\A,\Hi,D)$ to $(\B,\Hi',D')$, and a Morita equivalence of spectral triples from $(\B,\Hi',D')$ to $(\mathcal{C},\Hi'',D'')$, we construct a Morita equivalence from $(\A,\Hi,D)$ to $(\mathcal{C},\Hi'',D'')$. Clearly $\Hi'' = \mathcal{F} \otimes \Hi' = \mathcal{F} \otimes \mathcal{E} \otimes \Hi = \mathcal{G} \otimes \Hi$.

Denote the respective connections of the Morita equivalences by $\nabla_D$ and $\nabla_{D'}$.
We can then define an $\Omega_D^1$-connection $\nabla_{D''}$ on $\mathcal{G}$ by 
\begin{equation}\nabla_{D''} := \nabla_{D'} \otimes \text{Id}_{\E} + \text{Id}_{\mathcal{F}} \otimes \nabla_D.\label{eqn:sum connection}\end{equation}
To see that this is connection, notice that it satisfies the Leibniz rule for the $\A$-action, since the action of $\nabla_{D'}$ on $\mathcal{F}$ trivially commutes with the action of the algebra $\A$ on $\mathcal{E}$. 

We still need to show that this is indeed a map $\mathcal{G} \rightarrow \mathcal{G} \otimes \Omega_D^1$. 
For the $\text{Id}_{\mathcal{F}} \otimes \nabla_D$ part this is clear.

We have that $\nabla_{D'} \otimes \text{Id}_{\E}$ part is a priori a map $\mathcal{F} \otimes \E \rightarrow \mathcal{F} \otimes \Omega_{D'}^1 \otimes \E$, with $\Omega_{D'}^1$ the $\A'$ bimodule spanned by $\{a[D',b] | a,b\in \A'\}$. But since $[D',b] = [\nabla_{D},b]$ acting from the left on $\E$, and $\nabla_{D}$ is a map from $\E$ to $\E \otimes \Omega_D^1$, we can view this as an element of $\mathcal{F}\otimes \E \otimes \Omega_D^1$, hence of $\mathcal{G}\otimes \Omega_D^1$.

This connection is such that $[\nabla_{D''},c] = [D'',c]$ for all $c\in \mathcal{C}$, since $[D'',c] = [\nabla_{D'},c]$, and the second part of \eqref{eqn:sum connection} commutes with the $\mathcal{C}$-action.
Also, $[\nabla_{D''},a] = [D,a]$ for all $a\in\A$, since the first part of \eqref{eqn:sum connection} commutes with the $\A$-action, hence it is a well-defined $\Omega_D^1$-connection if $\nabla_D$ is. 
As a linear combination of two Hermitian connections, it is a Hermitian connection.
\end{proof}
\begin{remark}
 This construction is also compatible with a real structure, since we can just write out \eqref{eqn:new Dirac operator real} for both connections, and check that they are the same. Also, the real structure $J''$ can easily be described in terms of the real structure $J$, since
\[ J''(u\otimes s\otimes \xi \otimes \bar{t} \otimes \bar{v}) = v\otimes J'(s\otimes \xi \otimes \bar{t}) \otimes \bar{u} = v\otimes \otimes t J \xi \otimes \bar{s} \otimes \bar{u},\]
with $s,t \in \E$, $u,v\in \mathcal{F}$ and $\xi\in \Hi$. 
\end{remark}

An open question is when Morita equivalences of spectral triples are proper equivalences, that is not only reflexive and transitive, but also symmetric. It is well known this does not hold in all cases. For example, if $(\A,\C^k,D,J)$ is a real spectral triple with a finite dimensional algebra $\A$, by the results of~\cite{krajewski_classification_1998} the Dirac operator for a real spectral triple can be written as
\[D = \Delta + J \Delta J^{-1},\]
with $\Delta$ a $k\times k$ self-adjoint matrix in $\Omega^1_D$. Together with \eqref{eqn:self-equivalence}, this implies that there is an Morita self-equivalence with the spectral triple $(\A,\Hi,0,J)$.
Since $\Omega_0^1 = \{ 0\}$, there is no connection such that there is Morita equivalence of spectral triples from $(\A,\C^k,0,J)$ to $(\A,\C^k,D,J)$, hence the relation cannot be symmetric. With the same arguments, if a spectral triple can be written as a direct sum of other spectral triples, and one of these spectral triples is finite, then the relation cannot be symmetric. In the following section, we will present the first case where the Morita ``equivalence'' is provably an equivalence relation.

The proof of symmetry as in~\cite{zhang_projective_2010}*{Theorem 5.6} allows for a more general type of connections, taking values in $\E\otimes B(\Hi)$, not just $\E \otimes \Omega_D^1$, still satisfying the Leibniz rule \eqref{eqn:Leibniz rule} and Hermitian condition \eqref{eqn:Hermitian condition}. This implies in particular that in the case of a Morita self-equivalence of a spectral triple with a finite dimensional algebra $\A$, a connection can take the form of any bounded self-adjoint operator $B$.

\section{Morita equivalences of noncommutative tori}\label{sec:morita equivalences of tori}
From~\cite{li_strong_2004}*{Theorem 1.1} we know that two algebra $\A_{\theta_1}$ and $\A_{\theta_2}$ of smooth noncommutative $n$-tori are strongly Morita equivalent if and only if their matrices $\theta_1$ and $\theta_2$ lie in the same $SO(n,n| \Z)$ orbit. This is the group of $2n \times 2n$ matrices which leave the quadratic form $\sum_{i} x_i x_{n+i}$ invariant, have determinant $1$ and have integer entries. We can write elements of $SO(n,n|\Z)$ in the following block form:
\begin{equation}\begin{pmatrix} A & B\\ C & D\end{pmatrix}\label{eqn:sonn decomposition},\end{equation} with $A,B,C,D$ $n\times n$ matrices with integer entries. 

We can identify several subgroups of this group. For the group $GL(n,\Z)$, if we take any $R \in GL(n,\Z)$ the matrix given by \begin{equation}\rho(R) = \begin{pmatrix} R & 0 \\ 0 & {\l(R^t\r)}^{-1}\end{pmatrix},\label{eqn:gln action}\end{equation} is in $SO(n,n|\Z)$. Also, for the additive group $\SkewM$ of skew-symmetric $n\times n$ matrices with integer entries, for any such matrix $N$ the matrix given by
\begin{equation}\nu(N) = \begin{pmatrix} \text{Id}_n & N\\ 0 & \text{Id}_n\end{pmatrix}\label{eqn:asn action}\end{equation}
is in $SO(n,n|\Z)$. Finally we have element $\sigma_2$, that acts on ${(x_i)}_{i=1}^{2n}$ by interchanging $x_1$ with $x_{n+1}$, $x_2$ with $x_{n+2}$ and leaving the rest invariant. 
The following is known:

\begin{lemma}[\cite{rieffel_morita_1999}]
The group $SO(n,n|\Z)$ is generated by the action of the elements $\rho(R), R\in GL(n,\Z)$ and $\nu(N), N\in \SkewM$ defined above, and the single element $\sigma_2$.\label{lem:generators of sonn}
\end{lemma}

The action of $SO(n,n|\Z)$ on an $n\times n$ skew-symmetric matrix $\theta$ is given by
\begin{equation} g \theta := (A \theta + B){(C \theta + D)}^{-1},\label{eqn:sonn action}\end{equation}
where $A,B,C$ and $D$ are the matrix components defined in \eqref{eqn:sonn decomposition}.
This action is only defined for the subset of the skew-symmetric matrices $\theta$ for which $C \theta + D$ is invertible for all possible $C,D$ such that $\begin{pmatrix} A & B\\ C & D\end{pmatrix}\in SO(n,n|\Z)$. We will restrict our attention to this set of $\theta$ and denote this set by $\boldsymbol{\Theta}^0_n$. This is a dense set of the second category in the set of all skew-symmetric $n\times n$ matrices, as shown in~\cite{rieffel_morita_1999}. In~\cite{li_strong_2004} the Morita equivalence is extended to skew-symmetric matrices $\theta$ such that $g\theta$ is not defined for all $g$.

We will describe the Morita equivalences of $C^*$-algebras induced by the generators of the group, following~\cite{rieffel_morita_1999}. Since Morita equivalence of $C^*$-algebras is well known to be a transitive relation, this is enough to calculate the equivalences for the entire group.

Given an skew-symmetric $n\times n$ matrix $N$ with integer entries, the action of $\nu(N)$ on the matrix $\theta$ can be deduced from \eqref{eqn:asn action} as $\theta' = \theta + N$. We see immediately that the algebras $\mathcal{A}_\theta$ and $\mathcal{A}_{\theta'}$ are isomorphic, since $e(\vx\cd\theta\vy) = e(\vx \cd(\theta + N)\vy)$. 

For every element $R \in GL(n,\Z)$ we have an element $\rho(R)$ of $SO(n,n|\Z)$ given by \eqref{eqn:gln action} and we see 
that $\rho(R)( \theta) = R\theta R^t$. The commutation relations become $U_{\vx} U_{\vy} = e(\vx \cd R \theta R^t \vy)U_{\vy}U_{\vx}$ so the algebra $\mathcal{A}_{\rho(R)\theta}$ is isomorphic to $\mathcal{A}_\theta$. 

The last generator $\sigma_2$ of $SO(n,n|\Z)$ is not an isomorphism, so the equivalence bimodule is a bit more involved. We will calculate here explicitly the effect of $\sigma_{2}$ on the action of the algebra $\A_\theta$, following the description given in~\cite{rieffel_morita_1999} and~\cite{rieffel_projective_1988}.

Set $q = n-2$. The Morita equivalence goes via the bimodule $\mathcal{S}(\R \times \Z^q)$,  the space of Schwartz functions over $\R\times \Z^q$:
\[ \mathcal{S}(\R) = \left\{ f\in C^{\infty} \middle|\  \sup_{x\in \R} |x^n \dfrac{d^k}{dx^k} f(x)| < \infty\ \forall k,n \in \N\right\},\]
with $\mathcal{S}(\Z)$ the obvious restriction of the above to the integers. We have $\mathcal{S}(\R \times \R) \simeq \mathcal{S}(\R) \times \mathcal{S}(\R)$.
 Write
\[\theta = \begin{pmatrix} \theta_{11} & \theta_{12}\\ \theta_{21} & \theta_{22}\end{pmatrix},\]
with $\theta_{11}$ the top $2\times 2$ part of $\theta$. If $\theta_{11}\neq 0$, this is an invertible matrix. Then 
\begin{equation}\sigma_2(\theta) = \begin{pmatrix} \theta_{11}^{-1} & -\theta_{11}^{-1} \theta_{12}\\ \theta_{21} \theta_{11}^{-1} & \theta_{22} -\theta_{21} \theta_{11}^{-1} \theta_{12} \end{pmatrix},\label{eqn:transformed theta}\end{equation}
Define
\[
J_o = \begin{pmatrix} 0 & 1\\ - 1 & 0\end{pmatrix},\quad
J_{2} = \begin{pmatrix} J_o & 0 & 0\\ 0 & 0 & I_q\\ 0 & -I_q & 0\end{pmatrix}.
\]
Write $\vx_q$ for the vector in $\R^q$ consisting last $q$ components of $\vx$ and let $\Too$ be any $2\times 2$ matrix such that 
% \[\Too^t  \Too = -\theta_{11}\] 
\[
\Too^t J_o \Too = -\theta_{11}\]
and $\Ttt$ any $q\times q$ matrix such that
\[
\Ttt^t - \Ttt = \theta_{22}.
\]
If $\theta_{11}$ is not zero, $J_o$ is similar to $-\theta_{11}$, and so there always exist such a $\Too$. Also, $\theta_{22}$ is skew-symmetric, so there exists a suitable $\Ttt$. 

We now follow~\cite{rieffel_morita_1999} in defining a right action of $\A_{\theta}$ and left action of $\A_{\theta'}$ on $\Sc(\R\times \Z^q)$ such that their actions commute.
\begin{lemma}\label{lem:right action}
If $\theta_{11}$ is invertible, i.e.\ not zero, a right action of the generators $U_{\vx}$ with $\vx \in\Z^n$ of $\A_\theta$ on $\mathcal{S}(\R\times\Z^q)$ is given by
\begin{equation}
 U_{\vx}^r f(t,p) = e\left( \left(\begin{pmatrix}\theta_{12}^t & \Ttt\end{pmatrix}\cd\vx\right) \cd p + {\l(\Too \cd\vx\r)}_1 \cd t\right) f(t+{(\Too\cd\vx)}_2, p + \vx_q)\label{eqn:right action},
\end{equation}
with ${(\Too\cd\vx)}_i$ the $i$-th component of $\Too\cd\vx$ and $\vx_q$ the vector composed of the last $q$ components of $\vx$. 
\end{lemma}
\begin{proof}
Define the homomorphism $T$ of $\Z^n$ into $\R \times \Z^q \times \R^q$ by the matrix action
\[T = \begin{pmatrix}    
\Too & 0\\
0 & I_q\\
\theta_{12}^t & \Ttt\end{pmatrix}.
\]

We see that $T$ maps $\Z^n \simeq \Z^{2}\times \Z^q$ onto a lattice in $\R^{2} \times \Z^q \times \R^q$. Given an element $(s,t,u,v)$ of $\R \times \R \times \Z^q \times \R^q$, one has the following action on $\Sc(\R \times \Z^q)$: 
\[\pi(s,t,u,v) f(x,p) = \langle t,x\rangle \langle v,p\rangle f(x+t,p+u),\]
where $\langle , \rangle$ is the natural pairing between a group and its dual, with values in $\T$. 

We see that value of $v$ only plays a role modulo $\Z$, hence we can reduce $T$ to map to $\R \times \R \times \Z^q \times \T^q$. But this means we can write $T$ as a map from $\Z^n$ into  $M \times \hat{M}$, with $M = \R \times \Z^q$ and $\hat{M}$ is the Pontryagin dual of $M$. Because $T^t J_2 T = -\theta$, we see that the generators $U_{\vx}$ with $\vx \in \Z^n$, and hence the whole algebra, give a right action of $\A_{\theta}$ on $\Sc(\R \times \Z^q)$.
\end{proof}
The way to find the endomorphism algebra, hence the strongly Morita equivalent algebra $\A_{\theta'}$, is to find an embedding $S$ of $\Z^n$ into $\R \times \R \times \Z^q\times \T^q$ such that $S(\Z^n)$ and $T(\Z^n)$ are dual lattices, i.e.\ the $S(\vx) \cdot J T(\vy) \in \Z$ for all $\vx,\vy\in\Z^n$.

\begin{lemma}\label{lem:left action}
The strongly Morita equivalent algebra $\A_{\theta'}$ acts on $\Sc(\R\times\Z^q)$ by 
\begin{align}
U_{\vx}^l f(t,p)
&= e\left( {\l( J_o {\l(\Too^t\r)}^{-1} \begin{pmatrix}I_{2} & -\theta_{12}\end{pmatrix} \cd\vx \r)}_2 \cd t + \l(\begin{pmatrix} 0 & T_{32}\end{pmatrix} \cd\vx\r)\cd p \right)\notag\\
& f(t+ J_o \l({\l(\Too^t\r)}^{-1} \begin{pmatrix}I_{2} & -\theta_{12}\end{pmatrix}_1\r), p + \vx_q).\label{eqn:left action}
\end{align}
\end{lemma}
\begin{proof}
Consider the map $S: \Z^n \rightarrow \R^{2} \times \Z^q \times \R^q$ given by the matrix
\[S = \begin{pmatrix}
J_o {\l(\Too^t\r)}^{-1} & -J_o {\l(\Too^t\r)}^{-1} \theta_{12}\\
0 & I_q\\
0 & -\Ttt^t
\end{pmatrix}.
\] 
Since $S^t \circ J_2 T \in \Z$,  this gives a lattice dual to $T(\Z^n)$ in $\R^{2} \times \Z^q \times \R^q$, and thus gives a commuting left action.
As can be checked, this is the algebra $\A_{\theta'}$ with $\theta' = \sigma_{2}(\theta)$.
\end{proof}

We thus have a Morita equivalence between the two algebra $\A_\theta$ and $\A_{\theta'}$, via the bimodule $\Sc(\R^p \times \Z^q)$, with the right action given by \eqref{eqn:right action} and the commuting left action given by \eqref{eqn:left action}.

\section{Morita equivalence of torus equivariant spectral triples}\label{sec:morita equivalence symmetric}
We now consider Morita equivalences of equivariant real spectral triples on the noncommutative $n$-torus, as described in Section~\ref{sec:equivariant spectral triples}. See~\cite{venselaar_2010} for details of the classification. We first determine the structure of the bimodule $\Omega^1_D$, and then deduce the general form of a connection from this structure.

\begin{lemma}\label{lem:omega is free}
The bimodule $\Omega^1_D = \text{span}\{a[D,b] | a,b\in \A_\theta\}$ is a free module of rank $n$, i.e.\ given by  
\[ 
\Omega_D^1 \simeq \A_\theta^{\oplus n}.
\]
Thus we have $\E \otimes_{\A} \Omega_D^1 \simeq \E^{\oplus n}$.
\end{lemma}
\begin{proof}
We know due to~\cite{venselaar_2010}*{Theorem A} that $D$ can be written as $\sum_i (\vtau_i \cd\vdelta) \otimes A_i + B$ with ${\{A_i\}}_{i=1}^n$ a set of $2^{\lfloor n/2\rfloor} \times 2^{\lfloor n/2\rfloor}$ matrices generating an irreducible representation of the Clifford algebra $Cl_{n,0}$, and $B$ a bounded operator commuting with the algebra action.
To see that the $A_i$ generate the whole module, first consider the module $\Omega_{D'}^1$ with $D' = \sum_i e_i \cd\vdelta A_i$. The modules $\Omega_{D'}^1$ and $\Omega_{D}^1$ are isomorphic by the action of an element $G \in GL(n,\R)$ such that $\sum_i G(\vtau_i) \cd\vdelta A_i = \sum_i e_i  \cd\vdelta A_i$ for all $i$. Such a $G$ exists, since the $\vtau_i$ are $n$ independent vectors in $\R^n$.The bounded operator $B$ does not play a role in the definition of $\Omega_{D}^1$, since it commutes with the algebra.

Since for all $j$, $[\delta_j,a] \in \A$ and the $A_i$ commute with the algebra action, we can write any $b = \sum_i a_i [D,c_i] \in \Omega_D^1$ as a sum $\sum_j b_j A_j$, with each $b_j = \sum_i a_i [\delta_j, c_i]$. Also, for any element $a\in \A$ we have  $a = a U_{e_i}^*[\delta_i, U_{e_i}]$, hence any $b = \sum_i b_i A_i$ gives rise to an element $b \in \Omega_D^1$. 

The $A_i$ are independent over $\A$, because
\[ \sum_j {(\pi(a_j) A_j)}^{\dagger}  \sum_i \pi(a_i) A_i  = \sum_i \pi(a_i^* a_i),\]
where the last equality uses the fact that the $A_i$ are generators of a Clifford algebra. This is a sum of positive operators $a_i^* a_i$, hence only zero if all $a_i$ are zero, thus the module is free.
\end{proof}
Using this decomposition, we can describe a connection $\nabla_D$ in terms of its components acting on the different subspaces of $\E\otimes\Omega^1_D$. In fact a more convenient description, in terms of the generators of the Hopf algebra $U(\mathfrak{t}^n)$, is available.

\begin{lemma}
The connection $\nabla_D$ on $\E = \Sc(\R\times \Z^{n-2})$ is a linear combination of $n$ connections $\nabla_i + b_i$, for which
\[[\sum_j {(\vtau_i)}_j \nabla_j, U_{\vx}^r] \psi \otimes A_i v = \psi \otimes [\vtau_i \cd\vdelta A_i,U_{\vx}] v,\]
and $b_i \in \text{End}_{\A}(\E)$, the Morita equivalent algebra.
\end{lemma}
\begin{proof}
Via the isomorphism induced by the action of $GL(n,\R)$ introduced in the proof of Lemma~\ref{lem:omega is free}, we see that in order to deduce the connection $\nabla_D$, it is enough to calculate the connections $\nabla_i$ where $\nabla_i$ is defined through the Leibniz rule as:
\[\nabla_i ( \psi  a ) v  =  (\nabla_i \psi) a v + \psi (\delta_i a)v,\]
with $\psi \in \E$, $v\in \Hi$, $a\in\A$ using the right action of $\A$ on $\E$. 
This determines the $\nabla_i$ up to an additive term $b_i$ commuting with the action of the algebra $\A_\theta$ on $\E$. Because $\E$ is assumed to be a Morita equivalence bimodule, we have $b_i \in \A_{\theta'}$.
\end{proof}

\begin{lemma}\label{lem:nabla_i}
Up to elements of the algebra $\A_{\theta'}$, the action of $\nabla_i$ on the space $\Sc(\R\times \Z^q)$ is given by:
\begin{align*}
\text{for $i=1,2$: }\nabla_i &= \Too^{-1} \cd\begin{pmatrix} t & \frac{1}{2 \pi i}\dfrac{\partial}{\partial t} \end{pmatrix},\\
\text{for $i=3,\ldots n$: }\nabla_i &= p.
\end{align*}
\end{lemma}
\begin{proof}
 By a simple calculation using \eqref{eqn:right action}, these connections give the right commutators with generators of the algebra $\A_{\theta}$.
\end{proof}

\begin{remark}
While these connections lead to an equivariant spectral triple, they themselves do not form the Hopf algebra for a equivariant torus action, since  
\[ [\nabla_1, \nabla_2] = {(\theta_{11}^{-1})}_{12}.\]
We only have equivariance with respect to an equivariant action on the last $q = n-2$ components.
\end{remark}

Recall that all Morita equivalences for $\A_{\theta}$ are given by the group $SO(n,n|\Z)$ if $\theta\in \boldsymbol{\Theta}^0_n$, the set of all antisymmetric matrices such that the $SO(n,n|\Z)$ action is defined for all elements of $SO(n,n|\Z)$. This is a dense set of the second category, see Section~\ref{sec:morita equivalences of tori}. 
\begin{lemma}\label{lem:transformed Dirac operator}
Given an equivariant noncommutative $n$-torus spectral triple $(\A_{\theta},\Hi,D,J)$ with $D$ as in \eqref{eqn:equivariant Dirac operator}, $\theta \in \boldsymbol{\Theta}^0_n$, and $\theta' = \sigma_2(\theta)$, then the Dirac operator $D'$ in the equivariant Morita equivalent spectral triple $(\A_{\theta'},\Hi',D',J')$ is given by:
\begin{align}
 D' &= \sum_i \left( \vtau^i \cdot \begin{pmatrix} \theta_{11}^{-1} & -\theta_{11}^{-1} \theta_{12}\\ 0 & I_q\end{pmatrix} \vdelta\right) \otimes \gamma_i + B'\notag\\
&= \sum_i \left(\begin{pmatrix} -\theta_{11}^{-1} & 0\\ \theta_{12}^t \theta_{11}^{-1} & I_q\end{pmatrix} \vtau^i \cdot \vdelta\right) \otimes \gamma_i + B',\label{eqn:transformed Dirac operator}
\end{align}
with $B'$ an arbitrary bounded self-adjoint operator commuting with the algebra $\A_{\theta'}$, and the $\gamma_i$ matrices generating an irreducible representation of the Clifford algebra $Cl_{n,0}$ and $\vtau_i$ as in \eqref{eqn:equivariant Dirac operator}.
\end{lemma}

\begin{proof}
The Dirac operator $D'$ on $\Hi'$ can be %
calculated using \eqref{eqn:new Dirac commutator}:
\[ [D', U_{\vx}](s \otimes \xi \otimes \bar{t}) = ([\nabla_D, U_{\vx}] s)\xi\otimes \bar{t}.\]
We know $\nabla_D$ in terms of the $\nabla_i$ calculated in Lemma~\ref{lem:nabla_i}:
\begin{align*}
\left[\nabla_i, U_{\vx}^l\right]  &= \Too^{-1} \cd\l( J_o {\l(\Too^t\r)}^{-1} \begin{pmatrix}I_{2} & -T_{31}^t\end{pmatrix} \cd\vx\r)\\
&= \theta_{11}^{-1} \begin{pmatrix}I_{2} & -\theta_{12}\end{pmatrix} \cd\vx.
\end{align*}
From this data we can now compute $D'$ as
\begin{equation} D' = \sum_i (\vtau_i' \cd\vdelta) \otimes A_i + B',\label{eqn:Morita equivalent Dirac operator}\end{equation}
where $[\vtau_i'\cd\vdelta,U_{\vx}] \psi\otimes v= [\nabla_i',U_{\vx}^l] \psi \otimes v$, and $\vdelta$ are the generators of the equivariant torus action.

The additional term we had for the connections, the $b_i \in \A_{\theta'}$ is seen to vanish when we demand a Dirac operator of the form \eqref{eqn:Morita equivalent Dirac operator}. Since an equivariant Dirac operator should commute with the Hopf algebra as in \eqref{eqn:equivariance dirac operator}, we have that $[\delta_j, b_i]=0$ for all $i,j$. However, in this case we have that $[b_i,U_{\vx}^l] = 0$ for all $i$ and $\vx\in \Z^n$, and so the contribution of the $b_i$ to the Dirac operator is a bounded operator commuting with the algebra, hence can be absorbed into $B'$.
\end{proof}

\begin{theorem}
For $\theta \in \boldsymbol{\Theta}^0_n$, if the algebras $\At$ and $\A (\T_{\theta'}^n)$ are strongly Morita equivalent and there exists a Morita equivalence $(\E,\nabla_D)$ of equivariant spectral triples from $(\At,\Hi,D,J)$ to $(\A (\T_{\theta'}^n),\Hi',D')$, then there exists a Morita equivalence $(\bar{\E},\nabla_{D'},)$ from  $(\A (\T_{\theta'}^n),\Hi',D',J')$ to $(\At,\Hi,D,J)$ such that the composition of the Morita equivalences gives a spectral triple unitary equivalent to the original spectral triple $(\At,\Hi,D,J)$. Thus Morita equivalence of equivariant torus spectral triples is an equivalence relation when $\theta \in \boldsymbol{\Theta}^0_n$.
\end{theorem}
\begin{proof}
We have for all $\theta \in \boldsymbol{\Theta}^0_n$ that the set of Morita equivalences is given by $SO(n,n|\Z)$, and by Lemma~\ref{lem:generators of sonn} we know that this group is generated by $GL(n,\Z)$, $\SkewM$ and the element $\sigma_2$. Since we know Morita equivalence of spectral triples is a transitive relation by Lemma~\ref{lem:morita transitive}, it is enough to construct inverses for the generators of this group. If $g\in GL(n,\Z)$ or $g\in \SkewM$, the Morita equivalence of algebras was actually an isomorphism. Since the algebras are isomorphic, the corresponding transformation of the Dirac operator is easily seen to be invertible. In fact, for the $\SkewM$ case, the Dirac operator is unchanged, and if  $R\in GL(n,\Z)$, the Dirac operator changes under $\rho(R)$ as:
\[ D' = (R^{-1} \vtau \cd\vdelta) A_i + B.\]

For the $\sigma_2$ generator, this follows from Lemma~\ref{lem:transformed Dirac operator}:
Since $\sigma_2 \circ \sigma_2 = \text{Id}$ on matrices, we see that the inverse of the Morita equivalence of algebras is given by applying $\sigma_2$ again. The corresponding Dirac operator is given by applying \eqref{eqn:transformed Dirac operator} twice, and using
\[\theta' = \sigma_{2}(\theta) = \begin{pmatrix} \theta_{11}^{-1} & - \theta_{11}^{-1} \theta_{12}\\ \theta_{21} \theta_{11}^{-1} & \theta_{22} - \theta_{21} \theta_{11}^{-1} \theta_{12}\end{pmatrix}.\]
We then see:
\begin{align*}
 \sigma_{2} (D')&= \sum_i \left( \begin{pmatrix} -\theta_{11} & 0\\  \theta_{12}^t\theta_{11}^{-1} \theta_{11} & I_q \end{pmatrix} \begin{pmatrix} -\theta_{11}^{-1} & 0\\ \theta_{12}^t \theta_{11}^{-1} & I_q\end{pmatrix} \vtau^i \cdot \vdelta\right) \otimes \gamma_i + B\\
&= \sum_i \left( \vtau^i \cd\vdelta\right) \otimes \gamma_i + B,
\end{align*}
hence the relation is symmetric. From Lemma~\ref{lem:morita transitive} we know that Morita equivalence of spectral triples is a transitive relation, and reflexivity is trivial via the identity map, thus Morita equivalence of equivariant spectral triples of the noncommutative $n$-torus is an equivalence relation.
\end{proof}
The fact that for equivariant spectral triples the Morita equivalences are provably invertible suggests that these type of spectral triples are in some way special among all spectral triples. In fact, even for noncommutative $n$-tori with $\theta \in \boldsymbol{\Theta}^0_n$, not even the dimension of the kernel of the Dirac operator in the Morita equivalent spectral triple is known if we drop the equivariance condition, see for example the discussion following~\cite{MR2366124}*{Corollary 5.2}. One would need to investigate if there is a canonical way to pick out these preferred spectral triples, even among spectral triples which do not admit an obvious equivariant action. Also the structure of the moduli space of noncommutative tori, especially the extra equivalences due to the $\sigma_2$-equivalence, merits investigation.

\bibliographystyle{abbrv}
\def\polhk#1{\setbox0=\hbox{#1}{\ooalign{\hidewidth
  \lower1.5ex\hbox{`}\hidewidth\crcr\unhbox0}}}
% \bib, bibdiv, biblist are defined by the amsrefs package.

\end{document}